\documentclass[12pt]{article}
\usepackage{amsfonts,amscd,amssymb}
\usepackage{theorem}
\newtheorem{definition}{Definition}[section]
\newtheorem{lemma}[definition]{Lemma}

{\theorembodyfont{\rmfamily}}
{\theorembodyfont{\rmfamily}}
\newtheorem{theorem}[definition]{Theorem}
{\theorembodyfont{\rmfamily}}
{\theorembodyfont{\rmfamily}}
{\theorembodyfont{\rmfamily}}

\newenvironment{proof}{{\it Proof.}}{\hfill $ \square $ \vskip 4mm}

\begin{document}

\title{An additive problem in connection with some quadratic real fields with class number one}

\author { Alexandru GICA}
\date{University of Bucharest, Faculty of Mathematics and Informatics\\Str. Academiei 14, Bucharest 1, Romania RO-010014\\ e-mail address: alexandru.gica@unibuc.ro}

\maketitle \small{ {\bf Abstract:}  
Our aim is to find all the prime numbers $p$ such that $p-x^2$ has at most two different prime factors, for all the odd integers $x$ such that $x^2< p$. We solve entirely the cases $p\equiv 1,5,7 \pmod 8$.  The case $p\equiv 3 \pmod 8$ is solved with one possible exception.

\textit{Dedicated to the memory of Florian Luca}

\textit{Key words and phrases:} Class numbers, Continued fractions.

\textit{Mathematics Subject Classification (2020):} 11R29, 11R11, 11A55.

\section{Introduction} 
$\;\;\;$The main goal of this paper is to prove the following.
\begin{theorem}
Let $p$ be an odd prime such that $p-x^2$ has at most two different prime factors for any odd integer $x$ such that $x^2< p$. If $p\equiv 7 \pmod 8$, then $p=7,23,47,167$. If $p\equiv 5 \pmod 8$, then $p=5,13,29,37,53,101,173,197,293,677$. If $p\equiv 1 \pmod 8$, then $p=17,41,73,89,97,113,137,233,257,353,593,857,1097,1217,1553,2273,5297$.  If $p\equiv 3 \pmod 8$ then, with one possible exception, $p=3,11,19,59,83,107,227.$ 
\end{theorem}

$\;\;\;$ We divided the proof of Theorem 1.1 in four cases. For $p\equiv 7 \pmod 8$ we prove that $p=a^2-2$, the class number of $\mathbb{Q}(\sqrt {p})$ is one and we use a result of J. Lee to solve this case. For $p\equiv 5 \pmod 8$, we prove that $p=a^2+4$ or $p=a^2+1$, the class number of $\mathbb{Q}(\sqrt {p})$  is one and use some classical results of A. Bir\'{o}  to solve this case. For $p\equiv 3 \pmod 8$  the class number of $\mathbb{Q}(\sqrt {p})$ is equal to 1. There are two possibilities. For $p=a^2+2$, we solve completely the problem (using again the result of Lee). The other possibility is $p=a^2+2b^2$, where $b$ is a prime number which equals $2a\pm 1$. This case is not completely solved. Since the length of the period of $\sqrt p$ (as a continuous fraction) is 6, using a result of Mollin and Williams, we found the list of the values of $p$ with one possible exception, which is excluded under the assumption of the Riemann Hypothesis. 

In the case $p\equiv 1 \pmod 8$, the proof is long and the computations are tedious, but, surprisingly enough, everything is elementary. We prove, in this case, that  $$p=17,41,73,89,97,113,137,233,257,353,593,857,1097,1217,1553,2273,5297.$$ If $p\neq 73,97$, we will prove that $p=(2^x+2^y-1)^2+2^{x+2}$ or $p=(2^y-2^x+1)^2+2^{x+2}$, for some positive integers $x,y$. After that, we will prove that $p$ is a quadratic residue modulo a prime $q$ in the set $\{5,7,13,17,109,241,433\}$. A lemma will ensure us that $p<16\cdot 433^2$. After some computations, we find all the possible values for $p$.

$\;\;\;$In the sequel we want to give some motivation for our work. We will use the standard notation $\omega$. If $n=p_1^{a_1}p_2^{a_2}...p_r^{a_r}$ is the standard decomposition of $n$ as product of primes (this means that $p_j$ are different prime numbers and $a_j$ are positive integers), then $\omega (n)=r$ (we consider $\omega (1)=0$). The prime numbers $p$ from Theorem 1.1 have the property $\omega (p-x^2)\leq 2$, for any odd integer $x$ such that $x^2< p$. We proved in [7] an analogue of Theorem 1.1.

\begin{theorem}
Let $p$ be an odd prime such that $p+x^2$ has at most two different prime factors for any odd integer $x$ such that $x^2<p$. If $p\equiv 5 \pmod 8$, then $p=5,13,37$. If $p\equiv 1 \pmod 8$, then $p=17,73,97,193$. If $p\equiv 3 \pmod 8$, then $p=3,11,19,43,67,163$.  If $p\equiv 7 \pmod 8$ then with one possible exception $$p=7,23,31,47,79,103,127,151,223,463,487,823,1087,1423.$$ Assuming a Restricted Riemann Hypothesis, the above list in the case $p\equiv 7 \pmod 8$ is complete.
\end{theorem}

The case $p\equiv 3 \pmod 8$ in Theorem 1.2 is in close connection with the celebrated theorem of Frobenius-Rabinowitsch, which characterize the imaginary quadratic fields with class number one, in terms of prime values of Euler's polynomial $E(X)=X^2+X+\frac{p+1}{4}$, on some consecutive integers. The case $p\equiv 5 \pmod 8$ in Theorem 1.1 was considered previously (with another formulation) by Mollin and Williams (see paper [9], where S. Louboutin gives a detailed description of the subject) and by myself in [4].

The interesting interplay between class numbers and the prime values of some quadratic polynomials, on some consecutive integers, when characterizing some real quadratic fields with class number one, is a classical area. The interested reader could find a presentation of these topics (results, history and the mathematicians who worked in this area) in the papers of Bir\'{o}, Chowla, Lee, Louboutin, Mollin, Williams and Yokoi.

{\bf Remark:} We proposed in 2025 a particular statement of Theorem 1.1 for the International Mathematical Olympiad (IMO).

\begin{theorem}
Let $p$ be a prime number such that $p=a^2+8$ and $\omega (p-x^2)\leq 2$ for any odd integer $x$ such that $x^2<p$. Then $p=17,89,233,1097$.
\end{theorem}

The Problem Selection Committee chose the problem for the Shortlist IMO 2025, as the Problem N8. The Committee changed a little bit the statement. They asked to prove that there are finitely many primes $p$ having the properties stated in Theorem 1.3
\section{The case $p\equiv 7 \pmod 8$}

$\;\;\;$ \begin{proof} We prove that in this case $p-x^2$ is twice a prime for any odd integer $x$ such that $x^2<p$. Let us suppose that there is an odd $x$ such that $x^2<p$ and $$p=x^2+2q^a,$$ for a prime $q$ and an integer $a\geq 2$. If $a=2$, then we obtain the contradiction
$$7\equiv p=x^2+2q^2\equiv 1+2 \equiv 3 \pmod 8.$$ Hence $a\geq 3$ and $4q^2<2q^a<p$. It follows that $(2q-x)^2<p$ and, according to the hypothesis, $$p=(2q-x)^2+2q^b.$$ It is easy to see that $q^2$ does not divide $p-(2q-x)^2$. This implies that $b=1$. We have 
$$2q=p-(2q-x)^2=2q^a-4q^2+4qx,$$
$$1=q^{a-1}-2q+2x=q(q^{a-2}-2)+2x\geq q+2x>1,$$ which is obviously a contradiction.
We proved that in this case $p-x^2$ is twice a prime number for any odd integer $x$ such that $x^2<p$ (the case $a=0, p=x^2+2$ is impossible; just consider a congruence modulo 8). We proved in [6] (see Theorem 3), a paper written together with Florian Luca, that this happens only for $p=7,23,47,167$. We first tackled the problem in 2004 (see[3]). Since in [3] and [6] we assumed that $p-a^2$ is twice a prime for any odd integer $a$ such that $a^2<p$ and $p-a^2$ is prime for any even integer $a$ such that $a^2<p$ (while here we know only the fact that $p-x^2$ is twice a prime for any odd integer $x$ such that $x^2<p$), we indicate here the steps of the proof. We prove first that $p=y^2-2$, in this case. By a classical result, there exist the positive integer $a,b$ ($a$ odd and $b$ even) such that $a^2<p$ and $p+a^2=2b^2$. Since $$p-a^2=2(b-a)(b+a)$$ should be twice a prime, it follows that $b=a+1$ and $p=2(a+1)^2-a^2=a^2+4a+2=(a+2)^2-2$. Let us suppose that $q$ is an odd prime such that $q\leq y-2$ and $p$ is a quadratic residue modulo $q$. Let $x$ an odd positive integer such that $x<q$ and $p\equiv x^2 \pmod q$. From what we proved above, it follows that $p=x^2+2q$. We obtain a contradiction since
$$p=x^2+2q<q^2+2q<(q+1)^2\leq (y-1)^2<y^2-2=p.$$ 
We proved that for any odd prime $q<y$, $p$ should be a quadratic non-residue modulo $q$. Since the Minkowski's constant 
for the quadratic field $K=\mathbb{Q}(\sqrt p)$ is $\sqrt p<y$ and 
$$N_{K/\mathbb{Q}}(y+\sqrt p)=2,$$ 
it follows that the ring of the integers of the quadratic field $K$ is principal. By a result of Jungyun Lee (see [8]), we know that $y\leq 19$. Only the values $y=3,5,7,13$ fit. Hence the only solutions in this case are $p=7,23,47,167$.
\end{proof}

\section{The case $p\equiv 5 \pmod 8$}

\begin{proof} From Fermat we know that $p$ is a sum of two squares, in this case. Since $p\equiv 5 \pmod 8$, it follows that there exist two odd positive integers $m,n$ such that 
$$p=m^2+4n^2.$$ Since $p-(2n-m)^2=4mn$ and $m,n$ are coprime, it follows (according to the hypothesis) that $m$ or $n$ should be 1. If both $m,n$ are 1, then $p=5$. 

{\bf The subcase $n=1,p=m^2+4>5$. } In this case, $p-x^2$ is 4 or four times a prime number for any odd integer $x$ such that $x^2<p$. Let us suppose that there is an odd positive integer $x$ such that $p=x^2+4q^a$, where $a\geq 2$ and $q$ is an odd prime. Since $4q^2<p$ and $x^2<p$, it follows that $(2q-x)^2<p$, and, according to the hypothesis, $p=(2q-x)^2+4q^b$. Since $q^2$ does not divide $p-(2q-x)^2$, it follows that $b=1$. We have
$$4q=p-(2q-x)^2=4q(q^{a-1}-q+x).$$ The only possibility is $a=2, x=1$. We obtained $p=m^2+4=1+4q^2$, which is a contradiction, since $p$ is uniquelly written as a sum of two squares (and $p\neq 5$). Hence, we proved that $p-x^2$ is 4 or four times a prime number for any odd integer $x$ such that $x^2<p$. We proved in 2008 (see [4], Theorem 1.2) that this happens only for the primes 
$$p=5,13,29,53,173,293.$$ The key point of the proof is that the class number of the field $\mathbb{K}=\mathbb{Q}(\sqrt {p})$ is one. According to Yokoi's Conjecture proved by A. Bir\'{o} (see [1]), this happens only for the aforementioned values of $p$.

{\bf The subcase $m=1,p=1+4n^2>5$. } Exactly as in the first case, we can prove that for any odd positive integer $x$ such that $x^2<p$, we have $p-x^2=4q^a$, where $q$ is an odd prime and $a=0,1,2$. Since $p$ could be written in only one way as a sum of two squares, it follows that the value $a=2$ occurs only for $x=1$, and the value $a=0$ does not occur. We proved in 2010 (see [5], Theorem 3.1) that this happens only for $p=37,101,197,677.$ The key point of the proof is that the class number of the field $\mathbb{K}=\mathbb{Q}(\sqrt {p})$ is one. According to Chowla's Conjecture proved by A. Bir\'{o}, this happens only for the aforementioned values of $p$.

\end{proof}

\section{ The case $p\equiv 3 \pmod 8$}

\begin{proof} Let $p\equiv 3 \pmod 8$ be a prime number as in the statement of Theorem 1.1 and let us consider $x$ an odd integer such that $x^2\leq p$. We know that $p-x^2=2q^a$. Let us suppose that $a\geq 3$. Then $4q^2<2q^a<p$, and $(2q-x)^2<p$. According to the hypothesis, $p=(2q-x)^2+2q$ (we used the fact that $q^2$ does not divide $p-(2q-x)^2)$). We have 
$$2q=p-(2q-x)^2=2q^a-4q^2+4qx.$$
Dividing by $2q$, we reach a contradiction:
$$1=q(q^{a-2}-2)+2x>1.$$
It is a classical result that the prime $p$ could be uniquely written as $$p=m^2+2n^2,$$ where $m,n$ are odd positive integers. According to the first remark of this section, $n=1$ or $n$ is a prime number.

{\bf The subcase $n=1, p=m^2+2>3$.} Let $x$ an odd integer such that $1\leq x <m$. We know that $p-x^2=2q^a$, where $q$ is an odd prime and $a=0,1,2$. The values $a=0,2$ does not occur since $p$ is uniquely written as $p=m^2+2n^2$. Hence $p-x^2$ is twice a prime number  for any odd $x$ such that $1\leq x<m$. We consider the field $\mathbb{K}=\mathbb{Q}(\sqrt {p})$, which has the Minkowski's constant $\sqrt p<m+1$. The ring of the integers for $K$ is $A=\mathbb{Z}[\sqrt {p}]$. 
We have $2A=P^2$, where $P$ is the maximal ideal $P=(m+\sqrt p)A$. Let $q$ an odd prime such that $q< m$. If $p$ is a quadratic residue modulo $q$, we have the equality
$$p=x^2+2q,$$ where $x$ is an odd integer such that $1\leq x<q$. We obtain the contradiction
$$q^2<m^2<p=x^2+2q\leq (q-2)^2+2q=q^2-2q+4.$$
Since $p>3$, it follows that 3 divides $m$. The only possibility for $m$ to be prime is when $m=3$ and $p=11$, which is a solution for our problem. Hence, for $p>11$, for every odd prime $q$ smaller than Minkowski's constant, $p$ is a quadratic non-residue modulo $q$. This means that $A$ is a principal ring. According to a theorem of Jungyun Lee (see [8]), there are only four prime numbers $p$ as above: $p=3,11,83,227$.

{\bf The subcase $p=m^2+2n^2$, $n$ prime, $2m>n$.} We have
$$p-(2n-m)^2=2n(2m-n).$$ According to the hypothesis, $2m-n$ should be 1. Hence 
$$p=m^2+2(2m-1)^2=9m^2-8m+2.$$
We consider again the field $\mathbb{K}=\mathbb{Q}(\sqrt {p})$, which has the Minkowski's constant $\sqrt p$.
We have $2A=P^2$, where $P$ is a principal maximal ideal (since the equation $x^2-py^2=-2$ has integer solutions for every prime number $p\equiv 3 \pmod 8$). Let $q$ an odd prime such that $q< \sqrt p$. If $p$ is a quadratic residue modulo $q$, we have the equality
$$p=x^2+2q^a,$$ where $x$ is an odd integer such that $1\leq x<q$ and $a$ is 1 or 2. If $a=1$, we obtain the contradiction
$$q^2<p=x^2+2q\leq (q-2)^2+2q=q^2-2q+4.$$ If $a=2$, then $q=n$, according to the uniqueness for $p$ as a sum of the type $m^2+2n^2$. In this case $nA=P_1P_2$, where $P_1,P_2$ are maximal ideals of norm $n$. Since 
$$p-(m+n)^2=n^2-2mn=n(n-2m)=-n,$$ it follows that $P_1$ and $P_2$ are principal ideals.
Hence, the only odd prime $q$ smaller than Minkowski's constant, such that $p$ is a quadratic residue modulo $q$ is $q=n$ and in this case the corresponding maximal ideals are principal. This means that $A$ is a principal ring. The key point now is the form of $\sqrt p$ as a continuous fraction:
$$\sqrt p = (3m-2; \overline {1,2,3m-2,2,1,6m-4}).$$
According to a theorem of Mollin and Williams (see [10]), there are at most 19 positive square-free numbers $d$ with such a property: the field $\mathbb{Q}(\sqrt d)$ has class number one and the length of the period of $\sqrt d$ as a continuous fraction is six; from the known 18 values, only four are prime numbers: $19, 59,107,131$. The only value which fits here is $p=59$. And there is at most one other value with this property. 

{\bf The subcase $p=m^2+2n^2$, $n$ prime, $2m<n$.} We have
$$p-(3m)^2=2(n-2m)(n+2m).$$ According to the hypothesis, $n-2m$ should be 1. Hence 
$$p=m^2+2(2m+1)^2=9m^2+8m+2.$$
We consider again the field $\mathbb{K}=\mathbb{Q}(\sqrt {p})$, which has the Minkowski's constant $\sqrt p$.
We have $2A=P^2$, where $P$ is a principal maximal ideal (we saw this before). Let $q$ an odd prime such that $q< \sqrt p$. If $p$ is a quadratic residue modulo $q$, we have the equality
$$p=x^2+2q^a,$$ where $x$ is an odd integer such that $1\leq x<q$ and $a$ is 1 or 2. We saw in the previous subcase that the case $a=1$ is impossible.
 If $a=2$, then $q=n$, according to the uniqueness for $p$ as a sum of the type $m^2+2n^2$. In this case, $nA=P_1P_2$, where $P_1,P_2$ are maximal ideals of norm $n$. Since 
$$p-(m+n)^2=n^2-2mn=n(n-2m)=n,$$ it follows that $P_1$ and $P_2$ are principal ideals.
Hence, the only odd prime $q$ smaller than Minkowski's constant, such that $p$ is a quadratic residue modulo $q$ is $q=n$ and in this case the corresponding maximal ideals are principal. This means that $A$ is a principal ring. The key point now is again  the form of $\sqrt p$ as a continuous fraction:
$$\sqrt p = (3m+1; \overline {2,1,3m,1,2,6m+2}).$$
According to the quoted theorem of Mollin and Williams ([10]), there are only finitely numbers $p$ with such a property (the ring $\mathbb{Z}[\sqrt p]$ is principal and the length of the period of $\sqrt p$ as a continuous fraction is six). Looking in the table, we see that the only values which fit are $p=19,107$. And there is at most one other value with this property. 
\end{proof}

\section{The case $p\equiv  1\pmod 8$}
We need the following Lemma.
\begin{lemma}
The prime $p\equiv 1 \pmod 8$ is like in Theorem 1.1. There is no odd prime $q$ such $16q^2<p$ and $p$ is a quadratic residue modulo $q$.
\end{lemma}
\begin{proof}
Suppose, by contradiction, that for an odd prime $q$ such that $16q^2<p$, there is an odd positive integer $m<q$ such that $p\equiv m^2\pmod q$ and $16q^2<p$ (if $m$ is even, we replace $m$ by $q-m$). We have  the following four equalities:
\begin{equation}
p=m^2+2^aq^x
\end{equation}
\begin{equation}
p=(2q-m)^2+2^bq^y
\end{equation}
\begin{equation}
p=(2q+m)^2+2^cq^z
\end{equation}
\begin{equation}
p=(4q-m)^2+2^dq^t,
\end{equation}
where $a,b,c,d,x,y,z,t$ are positive integers and $a,b,c,d\geq 3$.
 Suppose $x\geq 2$. It follows at once that $y=z=1$. Since 16  does not divide $8qm=(2q+m)^2-(2q-m)^2$, we get from (2) and (3) that $b=3$ or $c=3$. We obtain a contradiction, since
$$9q^2<p=(2q\pm m)^2+8q\leq (3q-2)^2+8q=9q^2-4q+4<9q^2.$$ Hence $x=1$. In the same way we prove  that $t=1$. Since $(4q-m)^2-m^2=8q(2q-m)$, it follows from (1) and (4) that $a=3$ or $d=3$. For obvious reasons, $d$ should be 3. Since 
$$16q^2<p=(4q-m)^2+8q\leq (4q-1)^2+8q=16q^2+1,$$ it follows that $m=1$ and $p=16q^2+1$. Therefore $x=2$ and this is a contradiction (we saw above that $x=1$).
\end{proof}

We will divide the proof, in this case, in several steps.

\begin{proof} {\bf First step: $p$ could be written as an expression in two powers of 2.}
We put $p$ between two consecutive squares. Consider first the subcase $(2k-1)^2<p<(2k)^2$. We will prove that \begin{equation}p=(2k-1)^2+2^u.
\end{equation}

Let us suppose that $p=(2k-1)^2+2^uq^a$, where $a$ is a positive integer and $q$ an odd prime number. Then $p$ is a quadratic residue modulo $q$ and 
$$4k^2>p=(2k-1)^2+2^uq^a\geq (2k-1)^2+8q.$$
From here we get $8q<4k-1$, $8q\leq 4k-4$, $4q\leq 2k-2$ and $16q^2<(2k-2)^2<(2k-1)^2<p$. This is a contradiction, according to the Lemma. We proved (5). Since $p<4k^2$, we have the inequalities $2^u<4k^2-(2k-1)^2=4k-1<4k,$ $2^{u-3}<\frac{k}{2}$. We have $(1+2^{u-2})^2<p$, since $$(1+2^{u-2})^2<(1+k)^2\leq (2k-1)^2<p.$$
Since $1+2^{u-2}$ is an odd number (don't forget that $u\geq 3$) and $(1+2^{u-2})^2<p$, we can apply the condition from the statement of Theorem 1.1, and we obtain that $p-(1+2^{u-2})^2$ has at most two prime factors (one of them being for sure 2). Since $2^u-(2^{u-2}+1)^2=-(2^{u-2}-1)^2$, we have 
\begin{equation}
p-(2^{u-2}+1)^2=4(k-2^{u-3})(k+2^{u-3}-1)
\end{equation}
Since $(k-2^{u-3})+2^{u-2}-1=(k+2^{u-3}-1)$, one of the numbers $(k-2^{u-3})$ and $(k+2^{u-3}-1)$ is odd and the other one is even. Suppose that the above numbers have a common odd prime factor $q$. From the Lemma we know that $16q^2<p$. Hence, the only possible case is when $q=(k-2^{u-3})$, $2q=(k+2^{u-3}-1)$ and $q=2^{u-2}-1$.
For $u=4$, we have $q=3,k=5$ and $p=9^2+2^4=97$, which is a solution for our problem. If $u\geq 5$, it follows that $u$ is odd (since $u-2$ is a prime number, $q$ being a Mersenne prime). We obtained a contradiction since we have two writtings for $p$ as a sum $b^2+2c^2$:
$$p=(2k-1)^2+2^u=(2^{u-2}+1)^2+8q^2.$$
For $p\neq 97$, we know that the numbers $(k-2^{u-3})$ and $(k+2^{u-3}-1)$ are coprime. Since $p-(2^{u-2}+1)^2$ has at most two prime factors, it follows from (6) that $(k-2^{u-3})=2^t$ or $(k+2^{u-3}-1)=2^t$. In the first case we have 
\begin{equation}
p=(2^{u-2}+2^{t+1}-1)^2+2^u
\end{equation}
and in the second case we have
\begin{equation}
p=(1-2^{u-2}+2^{t+1})^2+2^u.
\end{equation}

Consider now the subcase $(2k)^2<p<(2k+1)^2$. We will prove that \begin{equation}p=(2k+1)^2-2^u.
\end{equation}

Let us suppose that there is an odd prime $q$ such that $q$ divides $(2k+1)^2-p$. Then $p$ is a quadratic residue modulo $q$ and 
$$(2k)^2<p\leq (2k+1)^2-8q.$$
From here we get $8q<4k+1$, $8q\leq 4k$, $4q\leq 2k$ and $16q^2\leq (2k)^2<p$. This is a contradiction, according to the Lemma. We proved (9). Since $p>4k^2$, we have the inequalities $2^u<(2k+1)^2-(2k)^2=4k+1, 2^u\leq 4k$, $2^{u-2}\leq k$. We have $(2^{u-2}-1)^2<p$, since $$(2^{u-2}-1)^2\leq (k-1)^2\leq (2k)^2<p.$$
Since $2^{u-2}-1$ is an odd number (don't forget that $u\geq 3$) and $(2^{u-2}-1)^2<p$, we can apply the condition from the statement of Theorem 1.1, and we obtain that $p-(2^{u-2}-1)^2$ has at most two prime factors (one of them being for sure 2). Since $-2^u-(2^{u-2}-1)^2=-(2^{u-2}+1)^2$, we have 
\begin{equation}
p-(2^{u-2}-1)^2=4(k-2^{u-3})(k+2^{u-3}+1)
\end{equation}
Since $(k-2^{u-3})+2^{u-2}+1=(k+2^{u-3}+1)$, one of the numbers $(k-2^{u-3})$ and $(k+2^{u-3}+1)$ is odd and the other one is even. Suppose that the above numbers have a common odd prime factor $q$. From the Lemma we know that $16q^2<p$. Hence, the only possible case is when $q=(k-2^{u-3})$, $2q=(k+2^{u-3}+1)$ and $q=2^{u-2}+1$.
For $u=3$, we have $q=3,k=4$ and $p=9^2-2^3=73$, which is a solution for our problem. If $u\geq 4$, it follows that $u$ is even (since $u-2$ is a power of 2, $q$ being a Fermat prime). We have:
$$p=(2k+1)^2-2^u=(2k+1-2^{u/2})(2k+1+2^{u/2}).$$ Since $p$ is prime, $2k+1-2^{u/2}=1$, $2k=2^{u/2}\geq 2^{u-1}$ (we used the inequality $k\geq 2^{u-2}$). This is impossible since $u-1>u/2$ for $u\geq 4$.
For $p\neq 73$, we know that the numbers $(k-2^{u-3})$ and $(k+2^{u-3}+1)$ are coprime. Since $p-(2^{u-2}+1)^2$ has at most two prime factors, it follows from (10) that $(k-2^{u-3})=2^t$ or $(k+2^{u-3}+1)=2^t$. In the first case we have 
\begin{equation}
p=(2^{u-2}+2^{t+1}+1)^2-2^u
\end{equation}
and in the second case we have
\begin{equation}
p=(2^{t+1}-2^{u-2}-1)^2-2^u.
\end{equation}

We prove that equality (12) occurs only for $p=17$. Suppose that we have such an equality. Since $2^{u-2}\leq k=2^t-2^{u-3}-1$, we have $2^t\geq 2^{u-2}+2^{u-3}+1; t\geq u-1$. Suppose $u=2a$ ($a\geq 2$ since $u\geq 3$). From (12) and taking into account that $p$ is prime, we obtain
$$2^{t+1}-2^{2a-2}-1-2^a=1.$$ We obtained a contradiction since $t+1>2a-2\geq a\geq 2$ and $2^{t+1}-2^{2a-2}-2^a=2$. Hence $u=2a+1$ is odd. Suppose $t=2b+1$ is odd. From (12), we obtain
$$p=(2^{t+1}-2^{u-2}+1)^2-2^{2b+4}=(2^{t+1}-2^{u-2}+1-2^{b+2})(2^{t+1}-2^{u-2}+1+2^{b+2}).$$ Since $p$ is prime, we should have $2^{t+1}-2^{u-2}+1-2^{b+2}=1$. This means that 
$$2^{2b+2}-2^{b+2}-2^{u-2}=0.$$ This is impossible since $2b+1=t\geq u-1$ and $2b+1\geq b+2$. Hence $t$ is even, $t=2b$. From (12) we have
$$p=(2^{t+1}+2^{u-2}-1)^2-2^{t+u+1}=(2^{t+1}+2^{u-2}-1-2^{a+b+1})(2^{t+1}+2^{u-2}-1+2^{a+b+1}).$$
Since $p$ is prime, we should have $2^{t+1}+2^{u-2}-1-2^{a+b+1}=1$. This means that 
\begin{equation}2^{2b+1}+2^{2a-1}-2^{a+b+1}=2.
\end{equation}We have $2b=t\geq u-1=2a, b\geq a$, $2b+1\geq a+b+1$ and $2a-1\geq 1$. From these inequalities and from (13) it follows that $a=b=1, u=3,t=2,k=2,p=17$. 

Looking to formula (11), we see that it could be written also as
$$p=(2^{t+1}+2^{u-2}-1)^2+2^{t+3}.$$ This is the same as formula (7), switching $t+1$ and $u-2$. From al these remarks, we have to deal only with formulas (7) and (8).

{\bf The case $p=(2^{2x+1}+2^y-1)^2+2^{y+2}$.} Looking to formula (7), we claim that $t$ is even. If $t$ is odd, then 
$$p=(2^{u-2}+2^{t+1}-1)^2+2^u=(2^{u-2}+2^{t+1}+1)^2-2^{t+3}$$
and $p=(2^{u-2}+2^{t+1}+1-2^{\frac{t+3}{2}})(2^{u-2}+2^{t+1}+1+2^{\frac{t+3}{2}}).$ Since $p$ is prime, it follows that $(2^{u-2}+2^{t+1}+1-2^{\frac{t+3}{2}})=1$. This is impossible since $t+1\geq \frac{t+3}{2}$. Then $t$ is even, $t=2x$. We write $u-2=y$ and we have to analyse the case $p=(2^{2x+1}+2^y-1)^2+2^{y+2}$, where $x$ is a non-negative integer and $y$ a positive integer. 
We suppose, for the moment, that $p>433$. We will prove that there is an odd prime $q$ in the set $\{5,7,13,17,109,241,433\}$ such that $p$ is a quadratic residue modulo $q$. Suppose that this is false. Since $p$ is a quadratic non-residue modulo 5, we have six possibilities: $(x,y)\equiv (0,0), (0,1), (0,3), (1,0),(1,2),(1,3)$, where the congruence for $x$ is modulo 2 and the congruence for $y$ is modulo 4. Since $p$ is a quadratic non-residue modulo 17, we have 
\begin{equation}
(x,y)\equiv (0,3), (0,5),(1,6),(1,7),(2,3),(2,4),(3,2),(3,4),
\end{equation}
where the congruence for $x$ is modulo 4 and the congruence for $y$ is modulo 8. Since $p$ is a quadratic non-residue modulo 7, we found the following six possibilities for the pair $(x,y)$ modulo 3:
\begin{equation}
(x,y)\equiv (0,1), (0,2),(1,0),(1,1),(2,0),(2,1) \pmod 3.
\end{equation}
Combining (14) and (15), we have 48 possibilities for the pair $(x,y)$ ($x$ modulo 12 and $y$ modulo 24). We take into account that $2^{12}\equiv 1 \pmod {13}$ and $2^{24}\equiv 1 \pmod {241}$. Since we supposed that $p$ is a quadratic non-residue modulo 13 and 241, we remain only with five possibilities:
\begin{equation}
(x,y)\equiv (8,3), (5,7),(6,20),(2,4),(3,10),
\end{equation}
where the congruence for $x$ is modulo 12 and the congruence for $y$ is modulo 24. We analyse now 45 cases, when $x$ is considered modulo 36 and $y$ is considered modulo 72, taking into account that $2^{36}\equiv 1 \pmod {109}$ and $2^{72}\equiv 1 \pmod {433}$. In all these 45 cases, we obtain that $p$ is a quadratic residue modulo 109 or 433. Taking into account the Lemma, it follows that 
$$p<16\cdot 433^2.$$ Since $p=(2^{2x+1}+2^y-1)^2+2^{y+2}$, we obtain $2^{2x+1}+2^y-1<4\cdot 433=1732.$ From here, we get that $2x+1\in \{1,3,5,7,9\}$ and $y\in \{1,2,3,4,5,6,7,8,9,10\}$. Cheking all the cases, we find the good values 
$$p=17,41,89,113,137,257,353,593,1097,1217,1553,2273,5297.$$

{\bf The case $p=(2^{x}-2^y+1)^2+2^{y+2}$.} Looking to formula (8), we claim that $t\equiv u \pmod 2$. If $u\equiv t+3\pmod 2$, then

\begin{equation}
p=(-2^{u-2}+2^{t+1}+1)^2+2^u=(2^{u-2}-2^{t+1}+1)^2+2^{t+3}.
\end{equation}
Since the prime $p$ is uniquely written as $a^2+b^2$ or $c^2+2d^2$, it follows that $u=t+3$ and $p=1+2^u\geq 17$. It follows that $p$ is a quadratic residue modulo 17. 

In the sequel $t\equiv u \pmod 2$. We write $u-2=y,t+1=x$ and we have to analyse the case $p=(2^{x}-2^y+1)^2+2^{y+2}$, where $x,y$ are positive integers with different parity. According to (17), we can suppose that $x$ is even and $y$ is odd.
We suppose, for the moment, that $p>433$. We will prove that there is an odd prime $q$ in the set $\{5,7,13,17,109,241,433\}$ such that $p$ is a quadratic residue modulo $q$. Suppose that this is false. Since $p$ is a quadratic non-residue modulo 5, we have three possibilities: $(x,y)\equiv (0,1), (0,3), (2,1)\pmod 4$. Since $p$ is a quadratic non-residue modulo 17, we have 
\begin{equation}
(x,y)\equiv (4,1), (4,3),(2,5),(6,5)\pmod 8.
\end{equation}
Since $p$ is a quadratic non-residue modulo 7, we found the following six possibilities for the pair $(x,y)$ modulo 3:
\begin{equation}
(x,y)\equiv (0,0), (0,2),(1,2),(2,0),(2,1),(2,2) \pmod 3.
\end{equation}
Combining (18) and (19), we have 24 possibilities for the pair $(x,y)$ modulo 24. We take into account that $2^{12}\equiv 1 \pmod {13}$ and $2^{24}\equiv 1 \pmod {241}$. Since we supposed that $p$ is a quadratic non-residue modulo 13 and 241, we remain only three possibilities:
\begin{equation}
(x,y)\equiv (4,17), (20,1),(14,21) \pmod {24}.
\end{equation}
 We analyse now 27 cases, when $x,y$ are considered modulo 72, taking into account that $2^{36}\equiv 1 \pmod {109}$ and $2^{72}\equiv 1 \pmod {433}$. In all these 27 cases, we obtain that $p$ is a quadratic residue modulo 109 or 433. Taking into account the Lemma, it follows that 
$$p<16\cdot 433^2.$$ Since $p=(2^{x}-2^y+1)^2+2^{y+2}$, we obtain $|2^{x}-2^y+1|<4\cdot 433=1732.$ From here, we get that $x,y\in \{1,2,3,4,5,6,7,8,9,10,11\}$ and $x$ is even and $y$ odd. Cheking all the cases, we find the good values 
$$p=17,41,113,233,353,857,1217.$$
Puting together all the facts, we found that in the case $p\equiv 1 \pmod 8$ we have the following solutions for Theorem 1.1:
$$p=17,41,73,89,97,113,137,233,257,353,593,857,1097,1217,1553,2273,5297.$$

\end{proof}

\
{\large \bf References}

[1] A. Bir\'{o}, \textit{Yokoi's conjecture}, Acta Arith., {\bf 106} (2003), 85--104.

[2] A. Bir\'{o}, \textit{Chowla's conjecture}, Acta Arith., {\bf 107} (2003), 179--194.

[3] A. Gica, \textit{An additive problem}, An. Univ. Buc. Mat., {\bf 53} (2004), 229--234.

[4] A. Gica, \textit{Some strange primes}, Bull. Math. Soc. Sci. Math. Roumanie, {\bf 51(99)} (2008), 213--217.

[5] M. Epure, A. Gica, \textit{Principal quadratic real fields in connection with some additive problems}, Bull. Math. Soc. Sci. Math. Roumanie, {\bf 53(101)} (2010), 251--259.

[6] A. Gica, F. Luca, \textit{On the Diophantine equation $2^x=x^2+y^2-2$}, Functiones et Approximatio, {\bf 46} (2012), 109--116.

[7] A. Gica, \textit{Class numbers, Ono invariants and some interesting primes}, Indagationes Mathematicae, {\bf 35} (2024), 1249--1258.

[8] J. Lee, \textit{The complete determination of wide Richaud-Degert types which are not 5 modulo 8 with class number one}, Acta Arith. {\bf 140} (2009), 1--29.

[9] S. Louboutin, \textit{Prime producing quadratic polynomials and class numbers of real quadratic fields}, Canad. J.  Math. {\bf 42} (1990),315--341.

[10] R.A. Mollin, H.C. Williams, \textit{On a Determination of Real Quadratic Fields of Class Number One and Related Continued Fraction Period Length Less than 25}, Proc. Japan Acad., {\bf 67} (1991), 20--25.

\end{document}